\documentclass[a4paper,10pt]{article}
\usepackage[utf8]{inputenc}

\usepackage{color}

\usepackage{amssymb}
\usepackage{amsmath}
\usepackage{amsfonts}
\usepackage{amsthm}
\usepackage{url}
\usepackage{algorithm}
\usepackage{algorithmic}
\usepackage{hyperref}
\usepackage{longtable}

\newcommand{\Co}{\mathcal{C}}
\newcommand{\R}{\mathbb{R}}
\newcommand{\ZZ}{\mathbb{Z}}
\newcommand{\KK}{\mathbb{K}}
\newcommand{\NN}{\mathbb{N}}
\newcommand{\F}{\mathbb{F}}

\newcommand{\Gr}{\text{Gr}}

\newcommand{\bv}{\mathbf{v}}
\newcommand{\bx}{\mathbf{x}}
\newcommand{\by}{\mathbf{y}}
\newcommand{\bz}{\mathbf{z}}

\newcommand{\lt}{\text{lt}}

\newcommand{\supp}{\text{supp}}

\newcommand{\Cos}{\text{Co}}
\newcommand{\T}{\mathbf{T}}

\newtheorem{theorem}{Theorem}[section]
\newtheorem{lemma}[theorem]{Lemma}
\newtheorem{proposition}[theorem]{Proposition}

\theoremstyle{definition}

\theoremstyle{remark}
\newtheorem{remark}[theorem]{Remark}

\newenvironment{example}[1][Example \arabic{section}.\arabic{theorem}.]
{\begin{trivlist}\refstepcounter{theorem}
\item[\hskip \labelsep {\bfseries #1}]}{\hfill $\diamondsuit$\end{trivlist}}

\title{Graver Bases and Universal Gröbner Bases for Linear Codes}
\author{Natalia Dück, Karl-Heinz Zimmermann}

\begin{document}
 \maketitle

\begin{abstract}
Two correspondences have been provided that associate any linear code over a finite field with a 
binomial ideal. In this paper, algorithms for computing their Graver bases and universal Gröbner bases are given.
To this end, a connection between these binomial ideals and toric ideals will be established. 
\end{abstract}

\section{Introduction}

Gröbner bases have originally been introduced by Buchberger for the algorithmic solution 
of some fundamental problems in commutative algebra~\cite{buch1} and turned out to be a crucial concept for 
further advance in the field of computer algebra~\cite{adams, becker, cls-app, grp}. 

Linear codes with their additional algebraic properties, on the other hand, form an important subclass of 
error-correcting codes and their relevance is well established in the field of coding theory~\cite{macws, vlint}. 

Recently, it has been emphasized that linear codes over finite fields
can be described by binomial ideals given as a sum of a toric ideal
and a non-prime ideal~\cite{borges,borges2,mahkhz3}. 
In this way, a direct link between the two prospering subjects of linear codes and Gröbner bases has been provided.
In the binary case, this correspondence proved to hold important information about the code like its minimum distance
and its minimal support codewords, thus allowing for a new decoding method~\cite{borges,marmartinez}.
Additionally, it led to new insights into the algebraic structure of linear
codes and allowed the application of slightly modified results from the 
rich theory of toric ideals~\cite{marmartinez,duekhz2}. 
Later, yet another correspondence was given which associates linear codes with binomial ideals and
which solves the complete decoding problem in the non-binary case~\cite{martcodeideal,marquez_corbella}.
Central to all these applications is the computation of  reduced Gröber bases. 
  
In this paper, we will address the problem of computing the Graver basis and the universal Gröbner basis 
for both binomial ideals associated to a linear code. 
In particular, we will extend methods used for accomplishing these tasks for toric
ideals as expounded in~\cite[Chapter 7]{sturmfels}.

The essential ideas stem from~\cite{marmartinez}. However, the method provided here differs from the one proposed
in~\cite{marmartinez} 
(compare~Corollary 4.5 in~\cite{marmartinez} and Prop.~\ref{prop:GraverOrdinary} and~\ref{prop:GraverGen}), where 
also only the modular case is considered.

This paper is organized as follows. 
Section~2 introduces all notions and definitions required later on.
Section~3 shows how both ideals associated to linear codes can be computed from certain toric ideals 
by substitution of variables. Section~4 deals with the computation of Graver bases.
In the final section~5, an algorithm for computing the universal Gröbner basis from the Graver basis 
for the generalized code ideal is given. 
Additionally, a sufficient condition for primitive binomials not belonging to the universal Gröbner basis is provided 
and the special case of characteristic~2 is emphasized.


\section{Preliminaries}

This section will introduce the necessary concepts from commutative algebra and algebraic coding. We assume familiarity with the basic definitions and notions of monomial orders and Gröbner bases as introduced in~\cite{adams,cls}.

\subsection{Toric Ideals, Gröbner Bases and Graver Bases}

Write $\KK[\bx] = \KK[x_1,\ldots,x_n]$ for the commutative polynomial ring in $n$ indeterminates over 
an arbitrary field $\KK$ and denote the \textit{monomials\/} in $\KK[\bx]$ by 
$\bx^{u} = x_1^{u_1}x_2^{u_2}\cdots x_n^{u_n}$, where $u=(u_1,\ldots,u_n)\in\NN_0^n$.

For a given ideal $I\subset\KK[\bx]$ and a monomial order $\succ$ on $\NN_0^n$, we shall denote
the {\em leading ideal\/} of $I$ w.r.t.\ $\succ$ by $\lt_\succ(I)$ and the
 reduced Gröbner basis for $I$ w.r.t.\ $\succ$ by $\mathcal{G}_\succ(I)$.
For a given ideal $I$ only finitely many different reduced Gröbner bases exist, 
and their union is called the \textit{universal Gröbner basis} for $I$ which will be denoted by 
$\mathcal{U}(I)$~\cite{schwartz,sturmfels,weispfennig87}. 

If two different monomial orders $\succ$ and $\succ'$ on $\NN_0^n$ have the same leading ideal: $\lt_\succ(I)=\lt_{\succ'}(I)$, 
then the reduced Gröbner bases are also the same: $\mathcal{G}_\succ(I)=\mathcal{G}_{\succ'}(I)$~\cite{fukJensenFans}. 
This result can be further generalized by introducing the notion of \textit{weight vectors}. 
For any $\omega\in\R^n$ and any polynomial $f=\sum c_i\bx^{u_i}\in\KK[\bx]$, define the \textit{initial form} $\lt_\omega(f)$ of $f$ 
to be the sum of all terms $c_i\bx^{u_i}$ in $f$ such that the inner product $\omega\cdot u_i$ is maximal,
and for an ideal $I$ define its \textit{leading ideal associated to $\omega$} as
\begin{align}
 \lt_\omega(I)=\left\langle \lt_\omega(f)\mid f\in I\right\rangle.
\end{align}
Note that unlike to leading ideals w.r.t.\ a monomial order this ideal is not necessarily generated by monomials.
For a non-negative weight vector $\omega\in\R_{+}^n$ and a monomial order $\succ$ on $\NN_0^n$, the new term order
$\succ_\omega$ is defined by ordering monomials first by their $\omega$-degree and breaking ties using $\succ$, i.e.,
\begin{align}
 \bx^a\succ_\omega\bx^b\quad :\Longleftrightarrow\quad 
a\cdot\omega>b\cdot\omega\:\vee\: (a\cdot\omega=b\cdot\omega\:\wedge\: \bx^a\succ\bx^b).
\end{align}
For any non-negative weight vector $w\in\R_+^n$ and any monomial order $\succ$ on $\NN_0^n$,
$\lt_w(I)=\lt_\succ(I)$ if and only if $\lt_w(g)=\lt_\succ(g)$ for all $g\in\mathcal{G}_\succ(I)$~\cite[Lemma 2.10]{fukJensenFans}

A \textit{binomial} in $\KK[\bx]$ is a polynomial consisting of two terms, i.e., a binomial is
of the form $c_u\bx^u-c_v\bx^v$, where $u,v\in\NN_0^n$ and $c_u,c_v\in\KK$ are non-zero. 
A binomial is \textit{pure} if the involved monomials are relatively prime.
All binomials considered here will be pure and henceforth the prefix pure will be omitted. 
A \textit{binomial ideal} is an ideal generated by binomials.

A binomial $\bx^{u}-\bx^{v}$ in a binomial ideal $I$ is \textit{primitive} if there is no other binomial 
$\bx^{u'}-\bx^{v'}$ in $I$ such that $\bx^{u'}$ divides $\bx^{u}$ and $\bx^{v'}$ divides $\bx^{v}$. The set of all primitive binomials in $I$ is called the {\em Graver basis\/} for $I$ and is denoted by $\Gr(I)$. It is easy to show that the universal Gröbner basis for a binomial ideal $I$ is always a subset of the Graver basis, $\mathcal{U}(I)\subseteq\Gr(I)$.

Toric ideals form a specific class of binomial ideals and can be defined in several ways~\cite{robbiano}. 
One way to introduce them is by means of integer matrices~\cite{sturmfels}.
For an integer $d\times n$ matrix $A$, the \textit{toric ideal} associated to $A$ is defined as
\begin{align}
I_A=\left\langle \bx^u-\bx^v\mid Au=Av,\,u,v\in\NN_0^n\right\rangle.\label{eq:toricrep1}
\end{align}

Note that each vector $u\in\ZZ^n$ can be uniquely written as $u=u^+-u^-$ where $u^+,u^-$ have disjoint support and their entries are non-negative. 
For instance, the vector $u=(1,-2,0)$ splits into $u^+ = (1,0,0)$ and $u^- = (0,2,0)$.
In this way, the toric ideal $I_A$ can be expressed as
\begin{align}
I_A=\left\langle \bx^{u^+}-\bx^{u^-}\mid u\in\ker_\ZZ(A)\right\rangle.\label{eq:toricrep2}
\end{align}

\subsection{Linear Codes and Binomials Ideals}

Let $\F_q$ denote the finite field with $q$ elements where $q$ is a prime power. 
In what follows, whenever we write $q=p^r$, $p$ shall be a prime and $r$ a non-negative integer. 
A \textit{linear code} $\Co$ of length $n$ and dimension $k$ over $\F_q$
is the image of a one-to-one linear mapping from $\F_q^k$ to $\F_q^n$.  
Such a code $\Co$ is called an $[n,k]$ code whose elements are called \textit{codewords},
which are always written as row vectors~\cite{macws,vlint}.

A \textit{generator matrix} for an $[n,k]$ code $\Co$ is a $k\times n$ matrix $G$ over $\F_q$ 
whose rows form a basis for $\Co$ and a \textit{parity check matrix} $H$ is an $(n-k)\times n$ matrix over $\F_q$ 
such that a word $c\in\F_q^n$ belongs to $\Co$ if and only if $cH^T=0$. 

The \textit{support} of a word $u\in\F_q^n$, denoted by $\supp(u)$, is the set of coordinates $i\in \{1,\dots,n\}$ such that $u_i\neq 0$.
\medskip

Let $\Co$ be an $[n,k]$ code over the finite field $\F_q$, where $q=p^r$ is a prime power. We
associate the following two binomial ideals to this code. 

The \textit{ordinary code ideal} associated to the code $\Co$ is an ideal in the polynomial ring 
$\KK[\bx]=\KK[x_{11},\dots,x_{1r},x_{21}\dots,x_{nr}]$ given as a sum of binomial ideals~\cite{borges2,mahkhz3},
\begin{align}
 I(\Co)=I'(\Co)+I_p,
\end{align}
where 
\begin{align}
I'(\Co) = \langle \bx^c-\bx^{c'}\mid c-c'\in\Co \rangle
\end{align}
and
\begin{align}
I_p=\left\langle x_{ij}^p-1\mid 1\leq i\leq n,\:1\leq j\leq r\right\rangle.
\end{align}
Note that the components of the word $c\in\F_q^n$ in the exponent of the monomial $\bx^c$ 
are replaced by their canonical integer representations using 
the vector space isomorphism between $\F_q$ and $\F_p^r$.

The binomial $\bx^u-\bx^{u'}$ in the code ideal is said to {\em correspond\/} to the codeword $u-u'$. 
In contrast to the integral case, however, different binomials may correspond to the same codeword.
For example, the word $(1,1,0)$ in $\F_2^3$ can be written as $(1,1,0)=(0,1,0)-(1,0,0)$ 
or $(1,1,0)=(1,0,0)-(0,1,0)$. 

In order to define the other binomial ideal associated to $\Co$, 
let $\alpha$ be a primitive element of $\F_q$ and define the {\em crossing map\/} 
$$\blacktriangle:\F_q^n\rightarrow\ZZ^{n(q-1)}$$ by 
\begin{align*}
\mathbf{a}=(a_1,\dots,a_n) = (\alpha^{j_1},\dots,\alpha^{j_n})\mapsto(\mathbf{e}_{j_1},\dots,\mathbf{e}_{j_n}),
\end{align*}
where $\mathbf{e}_i$ is the $i$th unit vector of length $q-1$, $1\leq i\leq q-1$, and 
each zero coordinate is mapped to the zero vector of length $q-1$. 
For instance, consider the field $\F_4 =\{0,\alpha,\alpha^2=\alpha+1,\alpha^3=1\}$ and $n=2$.
The crossing map $\blacktriangle:\F_q^2 \rightarrow \ZZ^{6}$ assigns $(\alpha,1)$ to $100001$,
$(0,0)$ to $000000$, and $(\alpha^2,0)$ to $010000$.

The associated mapping $$\blacktriangledown:\ZZ^{n(q-1)}\rightarrow \F_q^n$$ is given as
\begin{align*}
 (j_{1,1},\dots,j_{1,q-1},j_{2,1},\dots,j_{n,q-1})\mapsto 
\left(\sum_{i=1}^{q-1}j_{1,i}\alpha^i,\dots,\sum_{i=1}^{q-1}j_{n,i}\alpha^i\right).
\end{align*}

The \textit{generalized code ideal} associated to the code $\Co$ is an ideal in the larger polynomial ring
$\KK[\bx]=\KK[\bx_1,\dots,\bx_n]$, where $\bx_j=(x_{j1},x_{j2},\dots,x_{j,q-1})$ for $1\leq j\leq n$, given as~\cite{martcodeideal}
\begin{align}
 I_+(\Co)=\left\langle\bx^{\blacktriangle a}-\bx^{\blacktriangle b}\mid a-b\in\Co\right\rangle
\label{eq:otherIc}.
\end{align}

A generating set for the code ideal $I_+(\Co)$ will contain both 
a generating set of the associated linear code as well as their scalar multiples
and an encoding of the additive structure of the field $\F_q$~\cite{martcodeideal,mahkhzX}.
The latter can be given by the ideal $I_q$ in $\KK[\bx]$ generated by the set
\begin{eqnarray}
\bigcup_{i=1}^n \left( \left\{ x_{iu}x_{iv}-x_{iw}\mid\alpha^u+\alpha^v=\alpha^w\right\}
	  \cup \left\{ x_{iu}x_{iv}-1\mid\alpha^u+\alpha^v=0 \right\}\right). \label{eq-relationsI}
\end{eqnarray}

We introduce the following shorthand notations: We write $\mathcal{U}(\Co)=\mathcal{U}\left(I(\Co)\right)$, respectively $\mathcal{U}_+(\Co)=\mathcal{U}\left(I_+(\Co)\right)$, for the universal Gröbner basis and $\Gr(\Co)=\Gr\left(I(\Co)\right)$, respectively $\Gr_+(\Co)=\Gr\left(I_+(\Co)\right)$, for the Graver basis.

\section{Deriving code ideals from toric ideals}

In this section, the generalized code ideal $I_+(\Co)$ will be related to a toric ideal.
Such a connection has already been established for the ordinary code ideal $I(\Co)$ in the case of a 
prime field~\cite[Remark 1 and Proposition 3.1]{marmartinez}. 
To see this, define for any prime number $p$ and any $m\times n$ matrix $A$ over $\F_p$ 
the extended $m\times (n+m)$ integer matrix
\begin{align}
 A(p)=\left(\begin{array}{c|c}\triangle A & pI_m\end{array}\right)
\end{align}
where $\triangle A$ is an $m\times n$ integer matrix such that $A=\triangle A\otimes_\ZZ\F_p$. 

\begin{proposition}{\cite[Remark 1 and Proposition 3.1]{marmartinez}}\label{prop:codeidealToric}
The code ideal $I(\Co)$ associated to an $[n,k]$ code $\Co$ over $\F_p$ with parity check matrix $H$ can be expressed as
\begin{align}
 I(\Co)=\left\{f(\bx,\mathbf{1})\mid f\in I_{H(p)}\right\}\subset\KK[\bx],
\end{align}
where $\mathbf{1}$ is the all-$1$ vector of length $n-k$ and 
$I_{H(p)}$ is the toric ideal in $\KK[\bx,\by]=\KK[x_1,\ldots,x_n,y_1,\ldots,y_{n-k}]$ associated to the integer matrix $H(p)$.
\end{proposition}

This result can be extended to linear codes over any finite field as follows. 
Consider the finite field $\F_q$ with $q=p^r$ and 
recall that for a fixed basis $B=\left\{b_1,\dots,b_r\right\}$
and for any $0\leq s\leq p-1$ the monomial $x_{ij}^s$ represents the element $sb_j$ at the the $i$th position.
For a matrix $H\in\F_q^{m\times n}$ with row vectors $h_1,\dots,h_m$ construct the matrix
$$H'=\begin{pmatrix} b_1h_1\\\vdots\\b_rh_1\\\vdots\\b_1h_m\\\vdots\\b_rh_m\end{pmatrix}\in\F_q^{rm\times n}$$
that consists of all multiplications of the row vectors with elements from the basis $B$.
Replace then each entry by its vector representation in $\F_p^r$ in order to obtain
the extended matrix $H_e\in\F_p^{mr\times nr}.$ Finally, define the $mr\times nr+mr$ integer matrix $H(q)$ to be
\begin{align}
 H(q)=\left(\begin{array}{c|c} \triangle H_e&pI_{mr}
            \end{array}\right),\label{eq:Hqordinary}
\end{align}
where $\triangle H_e$ is an $mr\times nr$ integer matrix such that $\triangle H_e\otimes_\ZZ \F_p=H_e$.

\begin{example}
 Consider the following $2\times 4$ matrix over the finite field $\F_4$ with the four elements
$\left\{0,1,\alpha,\alpha^2=\alpha+1\right\}$,
\begin{align*}
 H=\begin{pmatrix}\alpha&0&1&0\\\alpha^2&\alpha&0&1
   \end{pmatrix}.
\end{align*}
First, we construct for the fixed basis $\{1,\alpha\}$
\begin{align*}
 H'=\begin{pmatrix}\alpha&0&1&0\\\alpha^2&0&\alpha&0\\\alpha^2&\alpha&0&1\\1&\alpha^2&0&\alpha
   \end{pmatrix}
\end{align*}
and then using the isomorphism defined by $1\mapsto (1,0)$ and $\alpha\mapsto (0,1)$ we construct
\begin{align*}
 H_e=\begin{pmatrix}0&1&0&0&1&0&0&0\\1&1&0&0&0&1&0&0\\1&1&0&1&0&0&1&0\\1&0&1&1&0&0&0&1
   \end{pmatrix}.
\end{align*}
This gives us the integer matrix $H(4)=\left(\triangle H_e\left|\: 2I_{4}\right.\right)$.
\end{example}

\begin{proposition}\label{prop:ordICtoric}
 The (ordinary) code ideal $I(\Co)$ associated to the $[n,k]$ code $\Co$ over the field $\F_q$ with parity
check matrix $H\in\F_q^{n-k\times n}$ can be expressed as
\begin{align}
 I(\Co)=\left\{f(\bx,\mathbf{1})\left|f\in I_{H(q)}\right.\right\},
\end{align}
where $\mathbf{1}$ is the all-one vector of length $(n-k)r$ and $I_{H(q)}$ is the toric ideal in the ring
$\KK[\bx,\by]=\KK[x_{11},\dots,x_{nr},y_{1},\dots,y_{(n-k)r}]$ associated to the integer matrix $H(q)$ defined
according to~(\ref{eq:Hqordinary}).
\end{proposition}
\begin{proof}
 It is sufficient to consider only binomials. 
Since $I_{H(q)}$ is a toric ideal the following equivalency holds 
for any $a,b\in\ZZ^{nr}$ and $a',b'\in\ZZ^{(n-k)r}$,
\begin{align*}
 \bx^a\by^{a'}-\bx^{b}\by^{b'}\in I_{H(q)}\quad&\Longleftrightarrow\quad 
		      \triangle H_e (a-b)^\T\equiv\mathbf{0}\mod p.
\end{align*}
Because $\triangle H_e$ is such that $\triangle H_e\otimes_\ZZ \F_p=H_e$, we easily see that identifying $a-b$ with
its image under the mapping 
$\ZZ^{nr}\ni a-b\mapsto a-b\mod p\in\F_p^{nr}$ belongs to $\ker(H_e)$ if and only if 
$\bx^a\by^{a'}-\bx^{b}\by^{b'}\in I_{H(q)}$. Note that the kernels $\ker(H_e)$ and $\ker(H)$ are isomorphic by the
isomorphism $\F_p^r\cong\F_q$. The result follows. 
\end{proof}


A similar result holds for generalized code ideals. 

In the following, let $\alpha$ be the fixed primitive element for the finite field $\F_q$. 
Fix a basis $B=\left\{b_1,\dots,b_r\right\}$ for $\F_q$ as an $\F_p$-vector space and denote for each $1\leq i\leq r$
by $\pi_i:\F_q\rightarrow\F_p$ the projection 
$$\F_q\ni a=\sum_{j=1}^{r} a_jb_j\mapsto a_i\in\F_p.$$

For a matrix $H=\left(h_{ij}\right)\in\F_q^{m\times n}$ with entries $h_{ij}$, $1\leq i\leq m$ and $1\leq j\leq n$,
construct first the extended matrix $H'\in\F_q^{m\times n(q-1)}$ by multiplying each entry
with $\alpha^s$ for each $1\leq s\leq q-1$ and arranging them line-by-line,
\begin{align*}
 H'=\begin{pmatrix}
      \ddots&&&&&\\
      &\alpha h_{ij}&\alpha^2 h_{ij}&\dots&\alpha^{q-1} h_{ij}&\\
      &&&&&\ddots
    \end{pmatrix}.
\end{align*}
Apply then to each entry the projections $\pi_i$ for all $1\leq i\leq r$ and arrange them column by column 
in order to obtain the extended matrix $H_{+,e}\in\F_p^{mr\times n(q-1)}$,
\begin{align*}
 H_{+,e}=\begin{pmatrix}
     \ddots&&&&&\\
   &\pi_1\left(\alpha h_{ij}\right)&\pi_1\left(\alpha^2 h_{ij}\right)&\dots&\pi_1\left(\alpha^{q-1} h_{ij}\right)&\\
   &\pi_2\left(\alpha h_{ij}\right)&\pi_2\left(\alpha^2 h_{ij}\right)&\dots&\pi_2\left(\alpha^{q-1} h_{ij}\right)&\\
   &\vdots&\vdots&&\vdots&\\
   &\pi_r\left(\alpha h_{ij}\right)&\pi_r\left(\alpha^2 h_{ij}\right)&\dots&\pi_r\left(\alpha^{q-1} h_{ij}\right)&\\     
&&&&&\ddots
    \end{pmatrix}.
\end{align*}
Finally, define the integer $mr\times n(q-1)+mr$ matrix $H_+(q)$ to be
\begin{align}
 H_+(q)&=\left(\begin{array}{c | c}
                       \triangle H_{+,e} & p I_{rm}
                      \end{array}\right)\label{eq:Hq+},
\end{align}
where $\triangle H_{+,e}$ is an $mr\times n(q-1)$ integer matrix such that $\triangle H_{+,e}\otimes_\ZZ\F_p=H_{+,e}$.

\begin{example}
Consider the following $2\times 3$ matrix over the finite field $\F_9$ with the nine elements 
$\left\{0,\alpha,\alpha^2,\dots,\alpha^8\right\}$, where $\alpha$ is a primitive element 
satisfying $\alpha^2+\alpha+2=0$,
\begin{align*}
 H=\begin{pmatrix}\alpha^2&\alpha&0 \\ 0&0&\alpha^6 \end{pmatrix}.
\end{align*}
First, we construct the matrix
\begin{align*}
 H'=\left(\begin{array}{c|c|c}
     H_1'&H_2'& H_3'
    \end{array}\right)
\end{align*}
with
\begin{align*}
 H_1'&=\begin{pmatrix}
       \alpha^3&\alpha^4&\alpha^5&\alpha^6&\alpha^7&\alpha^8&\alpha&\alpha^2\\
       0&0&0&0&0&0&0&0
      \end{pmatrix},\\
 H_2'&=\begin{pmatrix}
       \alpha^2&\alpha^3&\alpha^4&\alpha^5&\alpha^6&\alpha^7&\alpha^8&\alpha^1\\
       0&0&0&0&0&0&0&0
      \end{pmatrix},\\
 H_3'&=\begin{pmatrix}
	0&0&0&0&0&0&0&0\\
       \alpha^7&\alpha^8&\alpha&\alpha^2&\alpha^3&\alpha^4&\alpha^5&\alpha^6\\
      \end{pmatrix}.
\end{align*}
Next we fix the $\F_3$-vector space basis $\left\{1,\alpha\right\}$ for $\F_9$ and compute
the Table~\ref{tab:exampleproj}
\begin{table}[h]
\begin{tabular}{c|cccccccc}
 $x$&$\alpha$&$\alpha^2$&$\alpha^3$&$\alpha^4$&$\alpha^5$&$\alpha^6$&$\alpha^7$&$\alpha^8$\\ \hline
 $\pi_1(x)$&$0$&$1$&$2$&$2$&$0$&$2$&$1$&$1$\\
 $\pi_2(x)$&$1$&$2$&$2$&$0$&$2$&$1$&$1$&$0$
\end{tabular}
\centering
\caption{Evaluation of $\pi_1$ and $\pi_2$ for the field $F_9$}
\label{tab:exampleproj}
\end{table}
in order to obtain
\begin{align*}
 H_{+,e}=\left(\begin{array}{c|c|c}
            H_{1,+,e}&H_{2,+,e}&H_{3,+,e}
           \end{array}\right),
\end{align*}
where 
\begin{align*}
H_{1,+,e}&=\left(\begin{array}{cccccccc}2&2&0&2&1&1&0&1\\
				      2&0&2&1&1&0&1&2\\
				      0&0&0&0&0&0&0&0\\
				      0&0&0&0&0&0&0&0\end{array}\right),\\
H_{2,+,e}&=\left(\begin{array}{cccccccc}1&2&2&0&2&1&1&0\\
				      2&2&0&2&1&1&0&1\\
				      0&0&0&0&0&0&0&0\\
				      0&0&0&0&0&0&0&0\end{array}\right),\\
H_{3,+,e}&=\left(\begin{array}{cccccccc}0&0&0&0&0&0&0&0\\
				      0&0&0&0&0&0&0&0\\
				      1&1&0&1&2&2&0&2\\
				      1&0&1&2&2&0&2&1\end{array}\right).
\end{align*}
This gives us the integer matrix $H_+(9)=\left(\triangle H_e\left|\: 3I_{4}\right.\right)$.
\end{example}

\begin{proposition}\label{prop:genICtoric2} 
The generalized code ideal $I_+(\Co)$ associated to the $[n,k]$ code $\Co$ over the field $\F_q$ with $q=p^r$ 
and with parity check matrix $H$ can be expressed as
 \begin{align}
  I_+(\Co)=\left\{f(\bx,\mathbf{1})\mid f\in I_{H_+(q)}\right\},
 \end{align}
where $\mathbf{1}$ is the all-one vector of length $(n-k)r$ and $I_{H_+(q)}$ is the toric ideal in the ring
$\KK[\bx,\by]=\KK[x_{11},\dots,x_{n,q-1},y_{1},\dots,y_{(n-k)r}]$ associated to the integer matrix $H_+(q)$ defined
according to~(\ref{eq:Hq+}).
\end{proposition}
\begin{proof}
It is sufficient to restrict to binomials. 

Let $a,b\in\ZZ^{n(q-1)}$ and write 
$$a-b=(c_{11},\dots,c_{1,q-1},c_{21},\dots,c_{2,q-1},\dots,c_{n,q-1})=(\mathbf{c}_1,\dots,\mathbf{c}_n).$$
For some $a',b'\in\ZZ^{rm}$ holds
\begin{align*}
 \bx^a\by^{a'}-\bx^b\by^{b'}\in I_{H_+(q)}\quad\Longleftrightarrow\quad 
				  \triangle H_{+,e}(\mathbf{c}_1,\dots,\mathbf{c}_n)^\T\equiv\mathbf{0}\mod p.
\end{align*}
Identify the $c_{ij}$'s with their images under the canonical mapping $\ZZ\rightarrow\F_p$.
The equality $\triangle H_{+,e} (\mathbf{c}_1,\dots,\mathbf{c}_n)^\T\equiv\mathbf{0}\mod p$ is true if and only if 
for all $0\leq s\leq r-1$ and $1\leq i\leq n-k$ holds
\begin{align}
\sum_{j=1}^{n} \left(\pi_s\left(\alpha h_{ij}\right)c_{j1}+\ldots+ 
\pi_s \left(\alpha^{q-1}h_{ij}\right)c_{j,q-1}\right)  = 0  
\quad\mbox{over }\F_p.\label{eq:aux2}
\end{align}
On the other hand, we have $H\blacktriangledown(\mathbf{c}_1,\dots,\mathbf{c}_n)^\T=\mathbf{0}$ if and only if
for all $1\leq i\leq n-k$,
\begin{align}
 0=\sum_{j=1}^{n}h_{ij}\left(\sum_{\ell=1}^{q-1}c_{j\ell}\alpha^\ell\right)
=\sum_{j=1}^n \left( \left(h_{ij}\alpha\right)c_{j1}+\cdots+\left(h_{ij}\alpha^{q-1}\right)c_{j,q-1}\right)\quad\mbox{ over }\F_q.
\label{eq:aux3}
\end{align}
Another equivalent formulation for this equation is in $\F_p$ via the projections $\pi_s$, $1\leq s\leq r$.
To be more precise, Eq.~(\ref{eq:aux3}) is true if and only if for all $1\leq s\leq r$ and $1\leq i\leq n-k$
Eq.~(\ref{eq:aux2}) is true.
\end{proof}

\begin{example}\label{ex:F4}
Take the $[3,2]$ code $\Co$ over $\F_4$ with parity check matrix
\begin{align*}
 H&=\begin{pmatrix}\alpha&\alpha^3&\alpha^2
   \end{pmatrix},
\end{align*}
where $\alpha$ is a primitive element satisfying $\alpha^2+\alpha+1=0$.

Each entry $h_{1j}$, $1\leq j\leq 3$, in the parity check matrix is replaced by the row vector $(h_{1j}\alpha),h_{1j}(\alpha^2),h_{1j}(\alpha^3)$ 
and the entries are expanded according to the $\F_2$-basis $\{\alpha,1\}$ of $\F_4$:
\begin{align*}
 &\left(\begin{array}{c|c|c}
\alpha^2\:\:\alpha^3\:\:\alpha & \alpha\:\:\alpha^2\:\:\alpha^3 & \alpha^3\:\:\alpha\:\:\alpha^2
\end{array}
\right)
=\left(\begin{array}{c|c|c}
\alpha+1 \:\: 1 \:\:\alpha & \alpha\:\:\alpha+1\:\: 1 & 1\:\:\alpha\:\: \alpha+1
\end{array}
\right).
\end{align*}
By projection, this gives the matrix
\begin{align*}
 H_+(4)=\left(\begin{array}{ccccccccc|cc}
1&0&1 & 1&1&0 & 0&1&1 & 2&0\\
1&1&0 & 0&1&1 & 1&0&1 & 0&2
\end{array}
\right).
\end{align*}
We compute the reduced Gröbner basis for the toric ideal $I_{H_+(4)}$ w.r.t.\ the lexicographic ordering 
to consist of the binomials 
\begin{align*}
\begin{array}{lll}
x_{11}-x_{33},   & x_{12}-x_{31},  & x_{13}-x_{32},\\
x_{21}-x_{32},   & x_{22}-x_{33},  & x_{23}-x_{31},\\
x_{31}^2-y_2,	 & x_{31}x_{32}-x_{33}, & x_{31}x_{33}-x_{32}y_2,\\
x_{31}y_1-x_{32}x_{33}, & x_{32}^2-y_1,  & x_{33}^2-y_1y_2.
\end{array}
\end{align*}
The substitution $\by\mapsto\mathbf{1}$ and a further Gröber basis computation lead to the 
following set which is easily seen to be the reduced Gröbner basis for the generalized code ideal $I_+(\Co)$,
\begin{align*}
\left\{\begin{array}{lllll}
x_{11}-x_{33},   & x_{12}-x_{32}x_{33}, & x_{13}-x_{32},       & x_{21}-x_{32},& x_{22}-x_{33}, \\
x_{23}-x_{32}x_{33}, & x_{31}-x_{32}x_{33}, & x_{32}^2-1, & x_{33}^2-1
\end{array}\right\}.
\end{align*}
\end{example}


\section{Computing the Graver basis}

In~\cite[Remark 3]{marmartinez} it has been pointed out that the Graver basis for the (ordinary) code ideal
associated to a linear code over a finite prime field can be computed as an elimination ideal of the $\ZZ$-kernel 
of the matrix
\begin{align}
 \begin{pmatrix}\triangle H&\mathbf{0}&p I_m\\I_n&I_n&\mathbf{0}
    \end{pmatrix}\in\ZZ^{(m+n)\times(2n+m)},\label{eq-auxmatrix}
\end{align}
where $\triangle H$ is such that $\triangle H\otimes_\ZZ\F_p=:H$ is a parity check matrix for the corresponding code.
 
Based on this, a uniform method for computing the Graver basis for the ordinary and the 
generalized code ideal is developed which makes use of the connection between both code ideals and toric ideals
established in the last section. 

\subsection{Generalization of the Lawrence Lifiting}

For each $m\times n$ integer matrix $\triangle H$, let $H=\triangle H\otimes_\ZZ\F_p$ and 
define the \textit{$p$-Lawrence lifting} of $\triangle H$ 
as the $(m+n)\times(2n+m)$ integer matrix
\begin{align}
  \Lambda(H)_p=\begin{pmatrix}\triangle H&\mathbf{0}&p I_m\\I_n&I_n&\mathbf{0}
    \end{pmatrix}. \label{eq:ordinpLawrence}
\end{align}
Consider the toric ideal $I_{\Lambda(H)_p}$ in the ring $\KK[\bx,\by,\bz]$ where $\bx=(x_1,\dots,x_n)$, 
$\by=(y_1,\dots,y_n)$ and $\bz=(z_1,\dots,z_m)$,
and define the ideal $I_{\Lambda(H)}$ in $\KK[\bx,\by]$ as
\begin{align}
 I_{\Lambda(H)}=\left\{g(\bx,\by,\mathbf{1})\left| g\in I_{\Lambda(H)_p}\right.\right\}.\label{eq:elimZ}
\end{align}

\begin{proposition}\label{prop:kerH}
 The ideal $I_{\Lambda(H)}$ is a binomial ideal and all pure binomials in $I_{\Lambda(H)}$ are
of the form $\bx^{u}\by^{v}-\bx^{v}\by^{u}$, where $u-v\in\ker(H)$.
\end{proposition}
\begin{proof}
Let $\{g_1,\dots,g_k\}$ be a generating set for $I_{\Lambda(H)_p}$. 
Then by definition, $\{g_1',\dots,g_k'\}$, where $g_i'(\bx,\by)=g_i(\bx,\by,\mathbf{1})$ for $1\leq i\leq k$, is a generating set for $I_{\Lambda(H)}$. 
Since $I_{\Lambda(H)_p}$ is generated by binomials, so is $I_{\Lambda(H)}$.

In view of the second assertion, consider a binomial $\bx^{u^+}\by^{v^+}-\bx^{u^-}\by^{v^-}$ in $\KK[\bx,\by]$. 
Then the following holds:
\begin{align*}
 \bx^{u^+}\by^{v^+}\!-\!\bx^{u^-}\by^{v^-}\in I_{\Lambda(H)}\,
  &\Leftrightarrow\, \exists c\in\ZZ^m:\, (u^+\!-\!u^-,v^+\!-\!v^-,c)\in\ker\left(\Lambda(H)_p\right)\\
					       &\Leftrightarrow\, u^+-u^-\in\ker_{\ZZ_p}(H)\,\wedge\,u^+-u^-=v^--v^+\\
					       &\Leftrightarrow\, u^+-u^-\in\ker_{\ZZ_p}(H)\,\wedge\,  u^+=v^-\wedge u^-=v^+.
\end{align*}
This gives the result.
\end{proof}

\begin{proposition}\label{prop:coincide}
For each binomial ideal $I$ in $\KK[\bx,\by]$ in which every binomial is of the form $\bx^a\by^b-\bx^b\by^a$,
the Graver basis, the universal Gröbner basis and every reduced Gröbner basis coincide.
\end{proposition}
\begin{proof}
The Graver basis is a Gröbner basis w.r.t.\ any monomial order since it contains the universal Gröbner basis.
Claim that it is also the reduced Gröbner basis w.r.t.\ an arbitrary monomial order.
Indeed, suppose there are binomials $\bx^a\by^b-\bx^b\by^a$ and $\bx^c\by^d-\bx^d\by^c$ in $\Gr(I)$, 
where $\bx^a\by^b$ and $\bx^c\by^d$ are the respective leading terms.  
If $\bx^a\by^b$ divides $\bx^c\by^d$, then $\bx^b\by^a$ will divide $\bx^d\by^c$ contradictory 
to $\bx^a\by^b-\bx^b\by^a$ being primitive. 
By the same argument the non-leading term in a primitive binomial is not divisible by 
the leading term of another primitive binomial.  This proves the claim.
 
Since $\mathcal{G}_\succ(I)=\Gr(I)$, 
the inclusions $\mathcal{G}_\succ(I)\subseteq\mathcal{U}(I)\subseteq\Gr(I)$ are in fact equalities
and the result follows.
\end{proof}

\subsection{Application to Code Ideals}

In what follows, let $\Co$ be an $[n,k]$ code $\Co$ over the finite field $\F_q$ with $q=p^r$ being a prime number
and fix a $\F_p$-vector space basis $B=\left\{b_1,\dots,b_r\right\}$ for the field $\F_q$ as well as a primitive
element $\alpha$.

First, we consider ordinary code ideals. For this, 
let $\bx$ denote the variables $x_{11},\dots,x_{nr}$ and let $\by$ denote the variables $y_{11},\dots,y_{nr}$.
Let $\triangle H_e$ be the integer $(n-k)r\times nr$ matrix defined according to~(\ref{eq:Hqordinary}) and
constructed from the parity check matrix $H$ of the code $\Co$. Let $\Lambda(H_e)_p$ be the
$p$-Lawrence lifting of $\triangle H_e$ and let $I_{\Lambda(H_e)}$ be the ideal obtained from the toric ideal
associated to the matrix $\Lambda(H_e)_p$ by substituting $\bz\mapsto\mathbf{1}$ according to Eq.~(\ref{eq:elimZ}).

\begin{proposition}\label{prop:GraverOrdinary}
Let $\mathcal{G}$ be the reduced Gröbner basis for the ideal $I_{\Lambda(H_e)}$ w.r.t.\ any monomial order.
The Graver basis for the ordinary code ideal
$I(\Co)$ associated to the code $\Co$ is given by the set
\begin{align}
  \Gr\left(\Co\right)
      =\left\{ \bx^{u}-\bx^{v}\left|\bx^{u}\by^{v}-\bx^{v}\by^{u}\in\mathcal{G}\right.\right\}.
\end{align}
\end{proposition}
\begin{proof}
By Prop.~\ref{prop:kerH} and Prop.~\ref{prop:ordICtoric} we obtain that
\begin{align*}
   \Gr\left(\Co\right)
      =\left\{\bx^{u}-\bx^{v}\left|\bx^{u}\by^{v}-\bx^{v}\by^{u}\in\Gr\left(I_{\Lambda(H_e)}\right)\right.\right\}.
\end{align*} 
Additionally, the Gröbner basis $\mathcal{G}$ equals 
the Graver basis for the ideal $I_{\Lambda(H_e)}$ by Prop.~\ref{prop:coincide}.
Therefore, the result follows.
\end{proof}

This result gives rise to Alg.~\ref{alg-graverbasis_ord} which computes the Graver basis for the
ordinary code ideal. This procedure makes use of the following macros: 
\begin{itemize}
 \item $\texttt{triangleHe}(H,B)$ applied to an $m\times n$ matrix $H$ over the finite field $\F_q$ and an 
       $\F_p$-vector space basis $B$ for $\F_q$ returns the $mr\times nr$ integer matrix $\triangle H_e$ constructed
       from the matrix $H$ according to~(\ref{eq:Hqordinary}).
 \item $\texttt{pLawrenceLift}(M)$ applied to an integer matrix $M$ and an integer (prime) number $p$ returns its $p$-Lawrence
	lifting.
 \item $\texttt{toricIdeal}(M,m,n,o)$ applied to an integer matrix $M$ and non-negative integers $m,n,o$ 
	returns a generating set for the toric ideal associated to this matrix
       in the ring $\KK[\bx,\by,\bz]$ with $\bx$, $\by$ and $\bz$ having the respective sizes $m,n,o$.
 \item $\texttt{substitute}(S,\bv\rightarrow\mathbf{1})$ applied to a set of polynomials $S$ and 
	a sequence of variables $\bv$ returns the set of polynomials from $S$ where all variables from $\bv$ have
	been substituted by $1$.
 \item $\texttt{groebnerBasis}(I,\succ)$ applied to a set of polynomials $I$ and a monomial order $\succ$ returns
	the reduced Gröbner basis for the ideal generated by this set w.r.t.\ the given order. 
\end{itemize}

\renewcommand{\algorithmicrequire}{\textbf{Input:}}
\renewcommand{\algorithmicensure}{\textbf{Output:}}

\begin{algorithm}[tb]
\caption{Computation of the Graver basis for the code ideal}
\begin{algorithmic}[1] 
\REQUIRE Finite field $\F_q$ with prime power $q=p^r$ and with a specified basis $B$, and a
	 $(n-k)\times n$ matrix $H$ over $\F_q$
\ENSURE Graver basis for the code ideal $I(\Co)$ associated to the $[n,k]$ code $\Co$ over $\F_p$ 
with parity check matrix $H$
\STATE $\triangle He=\texttt{triangleHe}(H,B);$
\STATE $\Lambda(H)_p=\texttt{pLawrenceLift}(\triangle He,p);$  
\STATE $I_{}=\texttt{toricIdeal}(\Lambda(H)_p,nr,nr,(n-k)r);$
\STATE $I_{\Lambda(H)}=\texttt{substitute}(I,\bz\rightarrow\mathbf{1});$
\STATE $G=\texttt{groebnerBasis}(I_{\Lambda(H)},\succ);$
\RETURN $\Gr(\Co)=\texttt{substitute}(G,\by\rightarrow\mathbf{1})$
\end{algorithmic}
\label{alg-graverbasis_ord}
\end{algorithm}

\begin{example}\label{ex:F3}
Consider the $[3,2]$ code $\Co$ over the field $\F_3$ with parity check matrix
\begin{align*}
 H=\begin{pmatrix}1&2&1\end{pmatrix}.
\end{align*}
Note that $\triangle H_e=\begin{pmatrix}1&2&1\end{pmatrix}$ and thus, the $3$-Lawrence lifting of the matrix 
$\triangle H_e$ is given by
\begin{align*}
  \Lambda(H_e)_3=\left(\begin{array}{ccc|ccc|c}
                1&2&1&0&0&0&3\\\hline
		&&&&&&\\
		&I_3&&&I_3&&\mathbf{0}\\
		&&&&&&
                \end{array}\right).
\end{align*}
Computing the reduced Gröbner basis w.r.t.\ the grevlex order 
of the corresponding toric ideal $I_{\Lambda(H_e)_3}$ and substituting
$z\mapsto 1$ yields a generating set for $I_{\Lambda(H_e)}$,
\begin{align*}
\big\{
&x_3^3+y_3^3,\:
x_2^3+y_2^3,\:
x_2x_3+y_2y_3,\:
x_1^3+y_1^3,\:
x_1x_3^2+y_1y_3^2,\:\\
&x_1^2x_3+y_1^2y_3,\:
x_1x_2+y_1y_2,\:
x_3y_2^2+x_2^2y_3,\:
x_3^2y_2+x_2y_3^2,\\
&x_3y_1+x_1y_3,\:x_1^2y_2+x_2y_1^2,\:
x_1y_2^2+x_2^2y_1,\:
x_1x_3y_2+x_2y_1y_3\big\}
\end{align*}
which is also the reduced Gröbner basis w.r.t.\ the same monomial order. Finally, making the substitution 
$\by\mapsto\mathbf{1}$ yields the Graver basis for the code $\Co$
\begin{align*}
 \Gr\left(\Co\right)&=\big\{
x_3^3+1,\:
x_2^3+1,\:
x_2x_3+1,\:
x_1^3+1,\:
x_1x_3^2+1,\:
x_1^2x_3+1,\:
x_1x_2+1,\\
&\quad\quad x_3+x_2^2,\:
x_3^2+x_2,\:
x_3+x_1,\:x_1^2+x_2,\:
x_1+x_2^2,\:
x_1x_3+x_2\big\}.
\end{align*}
\end{example}

Second, we consider generalized code ideals. For this, let $\bx$ denote the variables $x_{11},\dots,x_{n,q-1}$
and let $\by$ denote the variables $y_{11},\dots,y_{n,q-1}$. Let $\triangle H_{+,e}$ be the integer $mr\times n(q-1)$
matrix defined according to~(\ref{eq:Hq+}) and constructed from the parity check matrix $H$ of the code $\Co$.
Let 
\begin{align*}
 \Lambda\left(H_{+,e}\right)_p=\left(\begin{array}{ccc}
\triangle H_{+,e} & \mathbf{0} & pI_{mr} \\ 
I_{n(q-1)} & I_{n(q-1)} & \mathbf{0}
\end{array}
\right)
\end{align*}
be the $p$-Lawrence lifting of the matrix $\triangle H_{+,e}$ and let $I_{\Lambda\left(H_{+,e}\right)}$
be the ideal obtained from the toric ideal associated to the integer matrix $\Lambda\left(H_{+,e}\right)_p$
by substituting $\bz\mapsto\mathbf{1}$ according to Eq.~(\ref{eq:elimZ}).

\begin{proposition}\label{prop:GraverGen}
Let $\mathcal{G}$ be the reduced Gröbner basis for the ideal $I_{\Lambda\left(H_{+,e}\right)}$ 
w.r.t.\ any monomial order. The Graver basis for the generalized code ideal
$I_+(\Co)$ associated to the code $\Co$ is given by the set
\begin{align}
  \Gr_+\left(\Co\right)
      =\left\{ \bx^{u}-\bx^{v}\left|\bx^{u}\by^{v}-\bx^{v}\by^{u}\in\mathcal{G}\right.\right\}.
\end{align}
\end{proposition}
\begin{proof}
Using the same arguments as in the proof of Prop.~\ref{prop:GraverOrdinary}, 
this is a direct consequence of Prop.~\ref{prop:kerH},~\ref{prop:genICtoric2} and~\ref{prop:coincide}.
\end{proof}

This provides Alg.~\ref{alg-graverbasis_gen} which computes the Graver basis for the generalized code ideal. 
This procedure is essentially the same as Alg.~\ref{alg-graverbasis_ord}. The only difference is that it uses the
following macro:
\begin{itemize}
 \item $\texttt{triangleHe+}$ applied to an $m\times n$ matrix $H$ over the finite field $\F_q$ and an $F_p$-vector
       space basis $B$ for $\F_q$ returns the integer $mr\times n(q-1)$ matrix $\triangle H_{+,e}$ constructed from
       the matrix $H$ according to~(\ref{eq:Hq+}) wit $\pi_i$, $1\leq i\leq r$ are the projections w.r.t.\ the basis
       $B$.
\end{itemize}

\begin{algorithm}[tb]
\caption{Computation of the Graver basis for the generalized code ideal}
\begin{algorithmic}[1] 
\REQUIRE Fnite field $\F_q$ with prime power $q=p^r$ and with a specified basis B, 
	  and an $n-k\times n$ matrix $H$ over $\F_q$
\ENSURE Graver basis for the generalized code ideal $I_+(\Co)$ associated to the $[n,k]$ code $\Co$ over $\F_q$ 
with parity check matrix $H$
\STATE $\triangle He+=\texttt{triangleHe+}(H,B);$
\STATE $\Lambda(H)_p=\texttt{pLawrenceLift}(\triangle He,p);$  
\STATE $I_{}=\texttt{toricIdeal}(\Lambda(H)_p,n(q-1),n(q-1),(n-k)r);$
\STATE $I_{\Lambda(H)}=\texttt{substitute}(I,\bz\rightarrow\mathbf{1});$
\STATE $G=\texttt{groebnerBasis}(I_{\Lambda(H)},\succ);$
\RETURN $\Gr(\Co)=\texttt{substitute}(G,\by\rightarrow\mathbf{1})$
\end{algorithmic}
\label{alg-graverbasis_gen}
\end{algorithm}

\begin{example}\label{ex:F4_cont}
(Ex.~\ref{ex:F4} cont'd)
The 2-Lawrence lifting of the matrix $\triangle H_{+,e}$ gives the matrix
\begin{align*}
 \Lambda\left(H_{+,e}\right)_2=\left(\begin{array}{ccc ccc ccc|ccc|cc}
  1&0&\multicolumn{1}{c|}{1}& 1&1&\multicolumn{1}{c|}{0}& 0&1&1& &\mathbf{0}&& 2&0\\
  1&1&\multicolumn{1}{c|}{0}& 0&1&\multicolumn{1}{c|}{1}& 1&0&1& &\mathbf{0}&& 0&2\\ \hline
   &&&  &&& && &&&\\
   &&&  &I_{9}& &  & & & &I_{9}&&\mathbf{0}&\mathbf{0}\\
   &&&  &&& && &&&&
 \end{array}\right).
\end{align*}
Using this matrix the Graver basis for $I_+(\Co)$ can be computed by Alg.~\ref{alg-graverbasis_gen}.
This basis consists of $135$ binomials.
\end{example}

\begin{remark}
 If $\bx^u-\bx^{u'}$ is a primitive binomial for a certain binomial ideal, then clearly $\bx^{u'}-\bx^u$
is also primitive. The Graver basis provided by the Alg.~\ref{alg-graverbasis_ord} 
and~\ref{alg-graverbasis_gen} will contain all primitive binomial only up to scalar multiples.
In other words, it will contain either $\bx^u-\bx^{u'}$ or $\bx^{u'}-\bx^u$.  
\end{remark}


\section{Computing the Universal Gröbner basis}

In this section it will shown how the universal Gröbner basis for ordinary and generalized code ideals can be
computed from the Graver basis. Although, all the following results are applicable to both code ideals,
they will be stated only for generalized code ideals. 

\begin{lemma}\label{lem:notinGr}
 Let $I$ be a binomial ideal in $\KK[\bx]$ and let $\bx^u-\bx^{u'}$ be a binomial in $I$.
If there is a binomial $\bx^v-\bx^{v'}\in I$ such that either both monomials $\bx^v$ and $\bx^{v'}$ divide $\bx^u$
or both monomials $\bx^v$ and $\bx^{v'}$ divide $\bx^{u'}$, then $\bx^u-\bx^{u'}$ does not belong to any reduced 
Gröbner basis for the ideal $I$.  
\end{lemma}
\begin{proof}
Let $\succ$ be any monomial order and $\bx^u$ be the leading term of the binomial $\bx^u-\bx^{u'}$ w.r.t.\ this order.

First, let $\bx^v-\bx^{v'}$ be a binomial in $I$ such that both terms divide $\bx^u$. 
One of them, $\bx^v$ or $\bx^{v'}$, is the leading monomial w.r.t.\ $\succ$. 
But as both are proper divisors of $\bx^u$, 
$\bx^u-\bx^{u'}$ cannot belong to the reduced Gröbner basis w.r.t.\ $\succ$.

Second, let $\bx^v-\bx^{v'}$ be a binomial in $I$ such that both terms divide $\bx^{u'}$. 
Again, either $\bx^v$ or $\bx^{v'}$ is the leading monomial w.r.t.\ $\succ$ and thus belongs to the leading ideal.
However, this is a contradiction to $\bx^{u'}$ being a standard monomial.
\end{proof}

\begin{example} (Ex.~\ref{ex:F3} cont'd)
The Graver basis for the linear code $\Co$ over $\F_3$ with parity check matrix $H=\left(1\: 2\: 1\right)$
 is given by
\begin{align*}
 \Gr\left(\Co\right)&=\big\{
x_3^3+1,\:
x_2^3+1,\:
x_2x_3+1,\:
x_1^3+1,\:
x_1x_3^2+1,\:
x_1^2x_3+1,\:
x_1x_2+1,\\
&\quad\quad x_3+x_2^2,\:
x_3^2+x_2,\:
x_3+x_1,\:x_1^2+x_2,\:
x_1+x_2^2,\:
x_1x_3+x_2\big\}.
\end{align*}
Applying Lem.~\ref{lem:notinGr} we deduce that $x_1x_3+x_2$ does not belong to the universal
Gröbner basis because both terms of the primitive binomial $x_1+x_3$ divide $x_1x_3$. In fact, it can be shown that
$\mathcal{U}(\Co)=\Gr(\Co)\setminus\left\{x_1x_3+x_2\right\}$.
\end{example}

\begin{proposition}\label{prop:char2uniGB}
The universal Gröbner basis for the generalized code ideal associated to a 
linear code over a finite field with characteristic two consists of exactly
all those primitive binomials whose involved terms are both unequal to~$1$
with the exception of the binomials of the form $x_{ij}^2-1$.
\end{proposition}
\begin{proof}
Let $\Co$ be a $[n,k]$ code over a field $\F_q$ with characteristic $2$. 
Note that because of $x_{ij}^2-1\in I_+(\Co)$ for all $1\leq i\leq n$ and $1\leq j\leq q-1$, 
all primitive binomials must be squarefree.
 
First, we show that no primitive binomial with one term equal $1$ belongs to the universal Gröbner basis.
For this, let $\bx^c-1\in I_+(\Co)$ be a primitive binomial with $\deg(\bx^c)>1$.
Since $\bx^c\neq 1$, we can write $\bx^c=x_{ij}\bx^{c'}$.
Then $x_{ij}(\bx^c-1)\equiv \bx^{c'}-x_{ij}\mod x_{ij}^2-1$ shows that $\bx^{c'}-x_{ij}$
belongs to $I_+(\Co)$. Since $\bx^{c'}$ and $x_{ij}$ are both proper divisor of $\bx^c$, 
the binomial $\bx^c-1$ cannot belong to the universal Gröbner basis according to Lem.~\ref{lem:notinGr}.

Second, we show that any primitive binomial whose involved terms are both unequal to $1$ belong to the 
universal Gröbner basis. For this,
let $\bx^u-\bx^{u'}\in I_+(\Co)$ be a primitive binomial with $u,u'\neq\mathbf{0}$ 
and put $\deg(\bx^u)=:s$ and $\deg(\bx^{u'})=:t$ and assume that $s\geq t$
(see Prop.~\ref{prop:bothB}).

Suppose this binomial does not belong to the universal Gröbner basis and hence not 
to any reduced Gröbner basis.
Let $\succ$ be a monomial order that orders
$$\left\{x_{ij}\mid ij\notin\supp(u)\cup\supp(u')\right\}\succ
\left\{x_{ij}\mid ij\in\supp(u)\cup\supp(u')\right\}$$
and that compares the monomials in $\{x_{ij}\mid ij\in\supp(u)\cup\supp(u')\}$ by their $\omega$-degree,
where $\omega_{ij}=1$ for $ij\in\supp(u)$ and $\omega_{ij}=\frac{s-1}{t}$ for $ij\in\supp(u')$.
For this order, $\bx^u\succ\bx^{u'}$ because $u\cdot\omega=s>s-1=\frac{s-1}{t}t=u'\cdot\omega$.

Since the considered binomial lies in the ideal $I_+(\Co)$ it must be reduced to zero on division by 
 the reduced Gröbner basis $\mathcal{G}_\succ(I_+(\Co))$ w.r.t.\ $\succ$ and so
there must be a pure binomial $\bx^v-\bx^{v'}\in\mathcal{G}_\succ(I_+(\Co))$ with leading term $\bx^v$ 
 such that $\bx^v\mid\bx^u$.

Clearly, $\supp(v)\cap\supp(v')=\emptyset$ and $\supp(v)\subset\supp(u)$ and by the chosen monomial order, 
$\supp(v')\subseteq\supp(u)\cup\supp(u')$.
But $\supp(v')\not\subset\supp(u')$ since this would contradict 
the primitiveness of the binomial $\bx^u-\bx^{u'}$.
Additionally, $\supp(v')\not\subset\supp(u)$ because otherwise the binomial 
$\bx^{v+v'}-1\equiv\bx^{v'}(\bx^v-\bx^{v'})\mod I_q$ would also contradict the primitiveness.
In other words, the monomial $\bx^{v'}$ must involve variables from $\bx^u$ as well as $\bx^{u'}$.

Claim that $\supp(v)\cup\supp(v')=\supp(u)\cup\supp(u')$. 
Indeed, write
$\bx^{v'}$ as a product of monomials $\bx^{u_1}$ and $\bx^{u_2}$ such that $\supp(u_1)\subset\supp(u)$, 
$\supp(u_2)\subset\supp(u')$. Then
$\bx^{u_1}(\bx^{v}-\bx^{u_1+u_2})\equiv\bx^{u_1+v}-\bx^{u_2}\mod I_q$ belongs to $I_+(\Co)$
 and we have $\bx^{u_1+v}\mid\bx^u$ and $\bx^{u_2}\mid\bx^{u'}$. Since the binomial $\bx^u-\bx^{u'}$
is primitive, this is a contradiction unless $\bx^{u_1+v}=\bx^u$ and $\bx^{u_2}=\bx^{u'}$.
This proves the claim.

On the whole, $\bx^v-\bx^{v'}=\bx^{u_1}-\bx^{u_2}\bx^{u'}$ where $\bx^{u_1}\bx^{u_2}=\bx^{u}$, 
$u_1,u_2\neq\mathbf{0}$. 

Since $\deg(\bx^u)=s$ there must be an integer $i\geq 1$ such that $\deg(\bx^{u_1})=s-i$ and $\deg(\bx^{u_2})=i$.
But then  $$u_1\cdot\omega=s-i<s\leq s-1+i=u'\cdot\omega+u_2\cdot\omega=(u'+u_2)\cdot\omega.$$
shows that $\lt_\succ(\bx^{v}-\bx^{v'})=\bx^{v'}$, which is a contradiction. Hence, the binomial $\bx^u-\bx^{u'}$
has to belong to the universal Gröbner basis.
\end{proof}

For linear codes over a finite field with characteristic $2$, Prop.~\ref{prop:char2uniGB} provides an easy
way to obtain the universal Gröbner from the Graver basis.

\begin{example}
Revisit Ex.~\ref{ex:F4} and~\ref{ex:F4_cont}. The Graver basis consists of $135$ binomials and
by Prop.~\ref{prop:char2uniGB} we deduce that the universal Gröbner basis consists of $36$ binomials less.
\end{example}

For codes over a finite field with characteristic greater than $2$ a method similar to the one in~\cite{sturmfels}
for toric ideals can be applied in order to compute the universal Gröbner basis from the Graver basis. 

To this end, for a given non-negative weight vector $\omega\in\R_+^{n(q-1)}$
and an ideal $I$, denote by $\mathcal{G}_{\omega}(I)$ the reduced Gröbner basis for $I$ w.r.t.\ $\succ_\omega$,
where $\succ$ is some tie breaking monomial order.
For $u,u'\in\NN_0^{n(q-1)}$ define the cone
\begin{align}
 C[u,u']=\left\{\omega\in\R_+^{n(q-1)}\mid \omega\cdot u>\omega\cdot u'\:\wedge
\:\bx^{ u}-\bx^{ u'}\in \mathcal{G}_\omega(I_+(\Co))\right\}.\label{eq:coneC}
\end{align}
\begin{remark}
 The cone $C[u,u']$ defined in (\ref{eq:coneC}) is essentially the same as in~\cite{sturmfels}. In this setting,
however, it is restricted to the positive orthant. 
\end{remark}

\begin{lemma}{\cite[Proposition 1.11]{sturmfels}}\label{lem:weightsucc}
 For any monomial order $\succ$ and any ideal $I\subset\KK[\bx]$, there exists a non-negative integer vector 
$\omega\in\NN^{n(q-1)}$ such that $\lt_\omega(I)=\lt_\succ(I)$.
\end{lemma}

\begin{proposition}\label{prop:Conenonempty}
The primitive binomial $\bx^u-\bx^{u'}\in I_+(\Co)$ belongs to the universal Gröbner basis of $I_+(\Co)$ 
if and only if the cone $C[u,u']$ is non-empty.
\end{proposition}
\begin{proof}
 If the cone $C[u,u']$ is non-empty, then obviously $\bx^u-\bx^{u'}\in\mathcal{U}_+(\Co)$.

Suppose $\bx^u-\bx^{u'}$ belongs to the universal Gröbner basis of $I_+(\Co)$.
Then there is a monomial order $\succ$ such that $\bx^u-\bx^{u'}\in\mathcal{G}_\succ(I_+(\Co))$
with $\bx^u$ being the leading monomial and thus $\bx^{u'}$ being a standard monomial. 
According to Lem.~\ref{lem:weightsucc} there is a weight vector $\omega\in\R_+^{n(q-1)}$ such that
 $\lt_\omega(I_+(\Co))=\lt_\succ(I_+(\Co))$.
Therefore, $\bx^u\in\lt_\omega(I_+(\Co))$. Moreover, $\bx^u-\bx^{u'}\notin\lt_\omega(I_+(\Co))$ because otherwise
$$\bx^u-\left(\bx^u-\bx^{u'}\right)=\bx^{u'}\in\lt_\omega(I_+(\Co))=\lt_\succ(I_+(\Co))$$ which is a contradiction
to $\bx^{u'}$ being a standard monomial. This implies that $\omega\cdot u>\omega\cdot u'$ 
and thus that $C[u,u']$ is non-empty.
\end{proof}

For $u\in\NN_0^{n(q-1)}$ and a linear code $\Co$ define the sets
\begin{align}
 \Cos(u,\Co)=\Cos(u)= \left\{v\in\NN_0^{n(q-1)}\left|\blacktriangledown u-\blacktriangledown v\in\Co\right.\right\}
\end{align}
and
 \begin{align}
 \mathcal{M}(u)=
\left\{\omega\in\R_+^{n(q-1)}\left| \omega\cdot u<\omega\cdot v\:\forall v\in\Cos(u)\setminus\{u\}\right.\right\}.
\label{eq:Mu_ineq}
\end{align}

\begin{lemma}{\cite[Lemma 7.4]{sturmfels}}\label{lem:compCone} 
For $u,u'\in\NN_0^{n(q-1)}$,
 \begin{align}
  C[u,u']=\mathcal{M}(u')\cap\bigcap_{ij\in\supp(u)}\mathcal{M}(u-\mathbf{e}_{ij}).
 \end{align}
\end{lemma}
For a proof see~\cite{sturmfels}.
The set $\mathcal{M}(v)$ and thus the cone $C[u,u']$ can be computed from the Graver basis~\cite{sturmfels}.
To see this, note that
$\omega\in\mathcal{M}(u)$ implies $\bx^u\notin\lt_\omega(I_+(\Co))$.
Additionally,
\begin{align}
 \lt_\omega(I_+(\Co))=\left\langle \lt_\omega(f)\mid f\in\Gr_+(\Co)\right\rangle
\end{align}
and hence, we see that a monomial $\bx^u$ does not belong to the leading ideal $\lt_\omega(I_+(\Co))$ if and only
if for every primitive binomial $\bx^v-\bx^{v'}$ in $I_+(\Co)$ such that $\bx^v$ divides $\bx^u$,
$\lt_\omega(\bx^v-\bx^{v'})\neq \bx^v$, which is equivalent to $\omega\cdot v\leq\omega\cdot v'$.
But as $\mathcal{M}(u)$ is open, we see that $\mathcal{M}(u)$ is described by all such strict inequalities. 
This yields the following alternative description of the set $\mathcal{M}(v)$
\begin{align}
 \mathcal{M}(v)=\left\{\omega\in\R_+^{n(q-1)}\left| \forall\bx^u-\bx^{u'}\in\Gr\mbox{ with }\bx^u\mid\bx^v: 
\left[\omega\cdot u'>\omega\cdot u\right]\right.\right\}.\label{eq:MvAlternDescription}
\end{align}

Similar to~\cite[Corollary 7.9, Proof of Theorem 7.8]{sturmfels} we show that if $\bx^u-\bx^{u'}$ belongs to 
the universal Gröbner basis for $I_+(\Co)$, then so does $\bx^{u'}-\bx^u$.
Although this is true for toric ideals and binomial ideals associated to integer lattices~\cite{sturmfels,swz_lattice},
it does not hold for binomial ideals in general as the following example demonstrates.
\begin{example}
Consider the binomial ideal $I=\left\langle x^2-xy, y^2-xy\right\rangle\subset\KK[x,y]$.
The reduced Gröbner basis w.r.t.\ the lex order with $x\succ y$ is given by the set $\{xy-y^2,x^2-y^2\}$.
Hence, $xy-y^2$ with belongs to the universal Gröbner basis of $I$. Suppose $y^2-xy$ also belongs to the universal
Gröbner basis and thus to some reduced Gröbner basis $\mathcal{G}_\succ(I)$ with $y^2\succ xy$. Pick any
 weight vector $\omega=(\omega_1,\omega_2)\in\R_+^2$ that represents $\succ$. 
Clearly, $\omega_2>\omega_1$ and thus, $xy\succ x^2$. But as $xy-x^2\in I$, $xy$ cannot be a standard monomial,
which is a contradiction to our assumption that $y^2-xy$ belongs to any reduced Gröbner basis.
\end{example}

\begin{lemma}\label{lem:bothbinomials}
 The primitive binomial $\bx^u-\bx^{u'}$ in $I_+(\Co)$ with $u,u'\neq\mathbf{0}$ belongs to the universal Gröbner basis if there is a 
non-negative vector $\omega\in\R_+^{n(q-1)}$ such that
\begin{align*}
 \omega\cdot u'\leq\omega\cdot u<\omega\cdot v\quad\forall v\in\Cos(u)\setminus\{u,u'\}.
\end{align*}
\end{lemma}
\begin{proof}
Suppose that such a vector $\omega\in\R_+^{n(q-1)}$ exists.

Claim that
$\omega\in\bigcap_{ij\in\supp(u)}\mathcal{M}(u-\mathbf{e}_{ij})$, i.e., every proper divisor of the monomial $\bx^{u}$ is standard w.r.t.\ the weight vector $\omega$. Indeed, if for any $ij\in\supp(u)$ holds
$\omega\notin\mathcal{M}(u-\mathbf{e}_{ij})$, then there has to be 
a $v\in\Cos(u-\mathbf{e}_{ij})\setminus\{u-\mathbf{e}_{ij}\}$ such that
$\omega\cdot(u-\mathbf{e}_{ij})\geq\omega\cdot v$. This implies 
$\omega\cdot (v+\mathbf{e}_{ij})\leq \omega\cdot u$ with $v+\mathbf{e}_{ij}\in\Cos(u)\setminus\{u\}$.
By the premise we conclude that $v+\mathbf{e}_{ij}=u'$. But then the binomial $\bx^{u}-\bx^{u'}$ is not pure, which is a contradiction to its primitiveness. This proves the claim.

We consider first the case that $\omega\cdot u'<\omega\cdot u$. Then by the definition, 
$\omega\in\mathcal{M}(u')$
and the result then follows  by Prop.~\ref{prop:Conenonempty} and Lem.~\ref{lem:compCone}. 

Finally, consider the case $\omega\cdot u'=\omega\cdot u$. Let $\succ$ be any monomial order such that
$\{x_{ij}\left| ij\in\supp(u)\right.\}\succ\{x_{ij}\left| ij\in\supp(u')\right.\}$. Therefore, $\bx^u\succ_\omega\bx^{u'}$ and since every proper divisor of $\bx^u$ is standard, we see that this monomial is actually a minimal generator in $\lt_{\succ_\omega}(I_+(\Co))$. Additionally, $\bx^{u'}$ is a standard monomial w.r.t.\ $\succ_\omega$ because $\omega\cdot u'<\omega\cdot v$ for all
$v\in\Cos(u)\setminus\{u,u'\}$. This proves that $\bx^{u}-\bx^{u'}\in\mathcal{G}_\omega(I_+(\Co))$.
\end{proof}

\begin{proposition}\label{prop:bothB}
If the binomial $\bx^u-\bx^{u'}$ with $u,u'\neq\mathbf{0}$ belongs to the universal Gröbner basis for a
generalized code ideal, then the binomial $\bx^{u'}-\bx^u$ also belongs to the universal Gröbner basis.
\end{proposition}
\begin{proof}
 Suppose that $\bx^u-\bx^{u'}$ belongs to the universal Gröbner basis with leading term $\bx^u$. 
Hence, this binomial is pure and primitive and there is a monomial order $\succ$ 
such that $\bx^u-\bx^{u'}\in\mathcal{G}_\succ(I_+(\Co))$ and 
$\lt_\succ(\bx^u-\bx^{u'})=\bx^u$. 

By Lem.~\ref{lem:weightsucc} there is a weight vector $\omega\in\R_+^{n(q-1)}$
that represents $\succ$. Suppose all coordinates of $\omega$ are strictly positive (otherwise $\omega$
can be replaced by a nearby vector from the same Gröbner cone).
In particular, $\omega\cdot u>\omega\cdot u'$. 

Define the weight vector $\omega'\in\R_+^{n(q-1)}$ as follows: put $\omega'_{ij}=0$ for $ij\in\supp(u)$ and
$\omega'_{ij}=\omega_{ij}$ otherwise. Hence, $0=\omega'\cdot u<\omega'\cdot u'$.
Based on that define another weight vector
\begin{align*}
 \omega''=(\omega\cdot (u-u'))\omega'-(\omega'\cdot (u-u'))\omega.
\end{align*}
Note that $\omega''$ is non-negative since $\omega'\cdot(u-u')$ is a negative scalar and $\omega\cdot(u-u')$
is a positive scalar. By definition $\omega''\cdot (u-u')=0$ and so $\omega''\cdot u=\omega''\cdot u'$.

Claim that for all $v\in\Cos(u)\setminus\{u,u'\}$ holds $\omega''\cdot u<\omega''\cdot v$. Indeed, 
if $\omega\cdot v< \omega\cdot u$, then the binomial $\bx^u-\bx^v\in I_+(\Co)$ 
has leading term $\bx^u$. We conclude that $\supp(u)$ and $\supp(v)$ are disjoint because otherwise $\bx^u$
would have a proper divisor that belongs to $\lt_\succ(I_+(\Co))$.
This implies $\omega'\cdot v=\omega\cdot v$. Furthermore, $\omega\cdot v>\omega\cdot u'$ 
because $\bx^{u'}$ is a standard monomial. Hence,
$$\omega''\cdot v =\left((\omega-\omega')\cdot(u-u')\right)(\omega\cdot v)
>\left((\omega-\omega')\cdot(u-u')\right)(\omega\cdot u')=\omega''\cdot u'=\omega''\cdot u.$$
If $\omega\cdot v\geq\omega\cdot u$, then 
\begin{align*}
 \omega''\cdot v&=(\omega\cdot (u-u'))(\omega'\cdot v)-(\omega'\cdot (u-u'))(\omega\cdot v)\\
		&\geq (\omega\cdot (u-u'))(\omega'\cdot v)-(\omega'\cdot (u-u'))(\omega\cdot u)\\
		&>-(\omega'\cdot (u-u'))(\omega\cdot u)=\omega''\cdot u.
\end{align*}
This proves the claim. And in particular, we obtain that
$\omega''\cdot u=\omega''\cdot u'<\omega''\cdot v$ for all $v\in\Cos(u)\setminus\{u,u'\}$
and thus by Lem.~\ref{lem:bothbinomials}, $\bx^{u'}-\bx^u\in\mathcal{U}_+(\Co)$.
\end{proof}

A method for computing the universal Gröbner basis for code ideals
from its Graver basis is given by Alg.~\ref{alg:uniGB}.
This procedure is based on the similar algorithm for toric ideals~\cite[Algorithm 7.6]{sturmfels}
and its correctness follows from
Prop.~\ref{prop:Conenonempty}, Lem.~\ref{lem:compCone} and eq.~(\ref{eq:MvAlternDescription}).

The proposed algorithm makes use of the following subroutines:
\begin{itemize}
 \item $\texttt{swap}(a,b)$ applied to the variables $a$ and $b$ swaps the contents of these variables.
 \item $\bx^u\mid\bx^v$ applied to the monomials $\bx^{u}$ and $\bx^{v}$ returns $1$ if the monomial $\bx^u$ divides
	the monomial $\bx^v$ and $0$ otherwise.
  \item $\texttt{addRow}(A,a)$ applied to an $m\times n$ integer matrix $A$ and an integer row vector $a$ of size $m$
	returns the matrix $A$ extended by the additional row $a$.
 \item $\texttt{break}$ quits the current \texttt{for}-loop.
 \item $\texttt{empty}(A)$  applied to an $m\times n$ integer matrix $A$ tests whether the open cone defined 
	by $\{\omega\in\R_+^{n}\mid A\omega> 0\}$ is empty and returns in this case $1$ or otherwise $0$.
\end{itemize}

\begin{algorithm}[htb]
\caption{Computation of the universal Gröbner basis}
\begin{algorithmic}[1] 
\REQUIRE Graver basis $\Gr_+(\Co)$
\ENSURE Universal Gröbner basis $\mathcal{U}_+(\Co)$
\STATE $\mathcal{U}_+(\Co)=\Gr_+(\Co);$
\STATE $A=[\:]$;
\FORALL{$\bx^u-\bx^{u'}\in\Gr_+(\Co)$}
\IF{$|\supp(u')|<|\supp(u)|$}
\STATE $\texttt{swap}(u,u')$;
\ENDIF
\FORALL{$\bx^v-\bx^{v'}\in\Gr_+(\Co)$}
\STATE 	$a_{11}=\bx^v \mid \bx^u$; 
\STATE $a_{12}=\bx^v\mid\bx^{u'}$;
\STATE 	$a_{21}=\bx^{v'} \mid \bx^u$; 
\STATE $a_{22}=\bx^{v'}\mid\bx^{u'}$; 
\IF{$(a_{11}\wedge a_{12})\vee(a_{21}\wedge a_{22})$}
\STATE $\mathcal{U}_+(\Co)=\mathcal{U}_+(\Co)\setminus\{\bx^u-\bx^{u'}\};$
\STATE $\texttt{break};$
\ENDIF
\FORALL{$ij\in\supp(u)$}
\IF{$\bx^v\mid\bx^{u-\mathbf{e}_{ij}}$}
\STATE $A=\texttt{addRow}(A,v'-v)$;
\ELSIF{$\bx^{v'}\mid\bx^{u-\mathbf{e}_{ij}}$}
\STATE $A=\texttt{addRow}(A,v-v')$;
\ENDIF 
\ENDFOR

\IF{$a_{12}$}
\STATE $A=\texttt{addRow}(A,v'-v)$;
\ELSIF{$a_{22}$}
\STATE $A=\texttt{addRow}(A,v-v')$;
\ENDIF
\ENDFOR
\IF{$\texttt{empty}(A)$}
\STATE  $\mathcal{U}_+(\Co)=\mathcal{U}_+(\Co)\setminus\{\bx^u-\bx^{u'}\};$
\ENDIF
\ENDFOR
\RETURN $\mathcal{U}_+(\Co)$
\end{algorithmic}
\label{alg:uniGB}
\end{algorithm}

\begin{remark}
 The run-time of Alg.~\ref{alg:uniGB} depends on the size of the Graver basis and is in the worst-case
 $O\left(|\Gr_+(\Co)|^2\right)$.  
\end{remark}


\bibliographystyle{plain}
\bibliography{refBook,refPaper}

\begin{thebibliography}{10}

\bibitem{adams}
W.~Adams and P.~Loustaunau.
\newblock {\em An Introduction to Gröbner Bases}.
\newblock American Mathematical Society, 1994.

\bibitem{becker}
T.~Becker and V.~Weispfenning.
\newblock {\em Gröbner Bases -- A Computational Approach to Commutative
  Algebra}.
\newblock Springer, 1998.

\bibitem{robbiano}
A.M. Bigatti and L.~Robbiano.
\newblock {Toric ideals}.
\newblock {\em Mathematica Contemporanea}, 21:1--25, 2001.

\bibitem{borges}
M.~Borges-Quintana, M.~A. Borges-Trenard, P.~Fitzpatrick, and E.~Martinez-Moro.
\newblock {Gröbner bases and combinatorics for binary codes}.
\newblock {\em AAECC}, 19(5):393--411, 2008.

\bibitem{borges2}
M.~Borges-Quintana, M.A. Borges-Trenard, and E.~Martinez-Moro.
\newblock {On a Gröbner bases structure associated to linear codes}.
\newblock {\em J. of Discret. Math. Sci. and Cryptogr.}, 10(2):151--191, 2007.

\bibitem{buch1}
B.~Buchberger.
\newblock {\em {An Algorithm for Finding the Bases Elements of the Residue
  Class Ring Modulo a Zero Dimensional Polynomial Ideal}}.
\newblock PhD thesis, University of Innsbruck, 1965.

\bibitem{cls}
D.~Cox, J.~Little, and D.~O'Shea.
\newblock {\em Ideals, Varieties, and Algorithms}.
\newblock Springer, 1996.

\bibitem{cls-app}
D.~Cox, J.~Little, and D.~O'Shea.
\newblock {\em Using Algebraic Geometry}.
\newblock Springer, 1998.

\bibitem{duekhz2}
N.~Dück and K.-H. Zimmermann.
\newblock {Universal Gröbner Bases for Binary Linear Codes}.
\newblock {\em International Journal of Pure and Applied Mathematics},
  86(2):345--358, 2013.

\bibitem{fukJensenFans}
Komei Fukuda, Anders~N. Jensen, and Rekha~R. Thomas.
\newblock Computing gr{\"o}bner fans.
\newblock {\em Math. Comput.}, 76(260):2189--2212, 2007.

\bibitem{grp}
G.-M. Greuel and G.~Pfister.
\newblock {\em A Singular Introduction to Commutative Algebra}.
\newblock Springer, Berlin, 2002.

\bibitem{macws}
F.J. MacWilliams and N.J.A. Sloane.
\newblock {\em {The Theory of Error-Correcting Codes}}.
\newblock North Holland, New York, 1977.

\bibitem{marmartinez}
I.~Marquez-Corbella and E.~Martinez-Moro.
\newblock {Algebraic structure of the minimal support codewords set of some
  linear codes}.
\newblock {\em Adv. in Math. of Commun.}, 5:233--244, 2011.

\bibitem{martcodeideal}
I.~Marquez-Corbella, E.~Martinez-Moro, and E.~Suarez-Canedo.
\newblock {On the ideal associated to a linear code}.
\newblock {\em Adv. in Math. of Commun.}, 2012.
\newblock submitted.

\bibitem{marquez_corbella}
Irene Marquez-Corbella.
\newblock {\em Combinatorial Commutative Algebra Approach to Complete
  Decoding}.
\newblock PhD thesis, Institute of Mathematics, University of Valladolid, 2013.

\bibitem{mahkhz3}
M.~Saleemi and K.-H. Zimmermann.
\newblock {Linear codes as binomial ideals}.
\newblock {\em Int. J. of Pure and Appl. Math.}, 61:147--156, 2010.

\bibitem{mahkhzX}
M.~Saleemi and K.-H. Zimmermann.
\newblock {Gröbner bases for linear codes over GF(4)}.
\newblock {\em Int. J. of Pure and Appl. Math.}, 73(4):435--442, 2011.

\bibitem{schwartz}
Niels Schwartz.
\newblock Stability of gröbner bases.
\newblock {\em Journal of Pure and Applied Algebra}, 53(1–2):171 -- 186,
  1988.

\bibitem{sturmfels}
B.~Sturmfels.
\newblock {\em Gröbner Bases and Convex Polytopes}.
\newblock American Mathematical Society, 1996.

\bibitem{swz_lattice}
B.~Sturmfels, R.~Weismantel, and G.~M. Ziegler.
\newblock {Gröbner Bases of Lattices, Corner Polyhedra, and Integer
  Programming}.
\newblock {\em Beiträge zur Algebra und Geometrie (Contributions to Algebra
  and Geometry)}, 36(2):281--298, 1995.

\bibitem{vlint}
J.H. van Lint.
\newblock {\em Introduction to Coding Theory}.
\newblock Springer, Berlin, 1999.

\bibitem{weispfennig87}
Volker Weispfenning.
\newblock Constructing universal gröbner bases.
\newblock In {\em AAECC}, pages 408--417, 1987.

\end{thebibliography}

\end{document}